\documentclass[reqno]{amsart}
\usepackage{amssymb}
\usepackage{amsmath}
\usepackage{amsfonts}
\usepackage{color}
\usepackage{bbm}
\usepackage{stmaryrd}
\usepackage{cite}
\usepackage{tikz}
\usepackage{tikz-cd}
\usepackage{xcolor-solarized}
\usepackage{geometry}

\usetikzlibrary{calc,decorations.pathreplacing}
\usetikzlibrary{arrows,shapes}

\definecolor{myurlcolor}{rgb}{0,0,0.4}
\definecolor{mycitecolor}{rgb}{0,0.5,0}
\definecolor{myrefcolor}{rgb}{0.5,0,0}
\usepackage[pagebackref,draft=false]{hyperref}
\hypersetup{colorlinks,
linkcolor=myrefcolor,
citecolor=mycitecolor,
urlcolor=myurlcolor}

\newcommand{\beq}{\begin{equation}}
\newcommand{\eeq}{\end{equation}}
\newcommand{\Z}{\mathbb{Z}}
\newcommand{\N}{\mathbb{N}}
\newcommand{\Q}{\mathbb{Q}}

\newcommand{\C}{\mathbb{C}}
\renewcommand{\H}{\mathcal{H}}
\newcommand{\B}{\mathcal{B}}

\newcommand{\staralg}{*\textnormal{-}\mathsf{alg}_1}
\newcommand{\Calg}{\mathsf{C^*alg}_1} 
\newcommand{\colim}[1]{\mathrm{colim}\, D}
\newcommand{\spec}[1]{\mathrm{Spec}(#1)} 
\newcommand{\id}[1]{\mathrm{id}_{#1}}

\theoremstyle{plain}
\newtheorem{thm}{Theorem}[section]
\newtheorem*{thm*}{Theorem}
\newtheorem{lem}[thm]{Lemma}
\newtheorem{prop}[thm]{Proposition}
\newtheorem{cor}[thm]{Corollary}
\newtheorem*{cor*}{Corollary}

\newtheorem{prob}[thm]{Problem}
\newtheorem{defn}[thm]{Definition}

\theoremstyle{remark}
\newtheorem{ex}[thm]{Example}
\newtheorem{rem}[thm]{Remark}

\allowdisplaybreaks

\usepackage{todonotes}

\begin{document}
\sloppy

\setlength{\jot}{6pt}



\title{Curious properties of free hypergraph C*-algebras}

\author{Tobias Fritz}
\email{tfritz@pitp.ca}
\address{Max Planck Institute for Mathematics in the Sciences, Leipzig, Germany}

\keywords{free hypergraph C*-algebra, undecidability, Connes Embedding Problem, nonlocal games, first incompleteness theorem}

\thanks{We thank William Slofstra and Andreas Thom for discussions, the anonymous referee for helpful feedback, Wilhelm Winter for suggesting the term ``free hypergraph C*-algebra'', and especially Nicholas Gauguin Houghton-Larsen for copious help with set theory and incompleteness and for detailed feedback on a draft. Without his assistance and patience, we probably would have stayed completely oblivious to the subtleties discussed in Appendix~\ref{logic}. We also thank David Roberson and William Slofstra for helpful comments on a draft.}

\subjclass[2010]{Primary: 46L99, 03D80; Secondary: 81P13, 03F40.}

\begin{abstract}
	A finite hypergraph $H$ consists of a finite set of vertices $V(H)$ and a collection of subsets $E(H) \subseteq 2^{V(H)}$ which we consider as partition of unity relations between projection operators. These partition of unity relations freely generate a universal C*-algebra, which we call the \emph{free hypergraph C*-algebra} $C^*(H)$. General free hypergraph C*-algebras were first studied in the context of quantum contextuality. As special cases, the class of free hypergraph C*-algebras comprises quantum permutation groups, maximal group C*-algebras of graph products of finite cyclic groups, and the C*-algebras associated to quantum graph homomorphism, isomorphism, and colouring.

Here, we conduct the first systematic study of aspects of free hypergraph C*-algebras. We show that they coincide with the class of finite colimits of finite-dimensional commutative C*-algebras, and also with the class of C*-algebras associated to synchronous nonlocal games. We had previously shown that it is undecidable to determine whether $C^*(H)$ is nonzero for given $H$. We now show that it is also undecidable to determine whether a given $C^*(H)$ is residually finite-dimensional, and similarly whether it only has infinite-dimensional representations, and whether it has a tracial state. It follows that for each one of these properties, there is $H$ such that the question whether $C^*(H)$ has this property is independent of the ZFC axioms, assuming that these are consistent. We clarify some of the subtleties associated with such independence results in an appendix.
\end{abstract}

\maketitle

\tableofcontents

\section{Introduction}

The \emph{quantum permutation group}~\cite{quantum_perms,quantum_perms_survey} of a finite set of cardinality $n$ is the universal unital C*-algebra generated by a matrix of projections $(p_{ij})_{i,j=1}^n$ satisfying a certain system of equations, which state exactly that this matrix is formally unitary. Due to the idempotency and self-adjointness of each matrix entry, this unitarity turns out to be equivalent to the condition that the projections making up each column should form a partition of unity,
\beq
\label{colsum}
	\sum_{i = 1}^n p_{ij} = 1, \qquad \forall j=1,\ldots,n,
\eeq
and similarly for each row,
\beq
\label{rowsum}
	\sum_{j = 1}^n p_{ij} = 1, \qquad \forall i=1,\ldots,n.
\eeq
In order to obtain the quantum permutation group, one also needs to equip this C*-algebra with a suitable comultiplication and antipode. But since we are concerned only with the C*-algebra structure in this paper, we will not discuss these additional structures, but rather focus on the above partition of unity relations only. These can be conveniently illustrated by a hypergraph as in Figure~\ref{qpg}.

\begin{figure}
\begin{tikzpicture}
\node[draw,shape=circle,fill,scale=.5] at (0,0) {} ;
\node[draw,shape=circle,fill,scale=.5] at (0,1) {} ;
\node[draw,shape=circle,fill,scale=.5] at (1,0) {} ;
\node[draw,shape=circle,fill,scale=.5] at (1,1) {} ;
\node[draw,shape=circle,fill,scale=.5] at (0,4) {} ;
\node[draw,shape=circle,fill,scale=.5] at (0,5) {} ;
\node[draw,shape=circle,fill,scale=.4] at (1,4) {} ;
\node[draw,shape=circle,fill,scale=.4] at (1,5) {} ;
\node[draw,shape=circle,fill,scale=.4] at (4,0) {} ;
\node[draw,shape=circle,fill,scale=.4] at (4,1) {} ;
\node[draw,shape=circle,fill,scale=.4] at (5,0) {} ;
\node[draw,shape=circle,fill,scale=.4] at (5,1) {} ;
\node[draw,shape=circle,fill,scale=.4] at (4,4) {} ;
\node[draw,shape=circle,fill,scale=.4] at (4,5) {} ;
\node[draw,shape=circle,fill,scale=.4] at (5,4) {} ;
\node[draw,shape=circle,fill,scale=.4] at (5,5) {} ;
\node at (2.5,0) {$\cdots$} ;
\node at (2.5,1) {$\cdots$} ;
\node at (2.5,4) {$\cdots$} ;
\node at (2.5,5) {$\cdots$} ;
\node at (0,2.5) {$\vdots$} ;
\node at (1,2.5) {$\vdots$} ;
\node at (4,2.5) {$\vdots$} ;
\node at (5,2.5) {$\vdots$} ;
\node at (2.5,2.5) {$\ddots$} ;
\draw[thick,blue] (2.5,0) ellipse (3.4cm and .3cm) ;
\draw[thick,blue] (2.5,1) ellipse (3.4cm and .3cm) ;
\draw[thick,blue] (2.5,4) ellipse (3.4cm and .3cm) ;
\draw[thick,blue] (2.5,5) ellipse (3.4cm and .3cm) ;
\draw[thick,blue] (0,2.5) ellipse (.3cm and 3.3cm) ;
\draw[thick,blue] (1,2.5) ellipse (.3cm and 3.3cm) ;
\draw[thick,blue] (4,2.5) ellipse (.3cm and 3.3cm) ;
\draw[thick,blue] (5,2.5) ellipse (.3cm and 3.3cm) ;
\end{tikzpicture}
\caption{Quantum permutation groups are free hypergraph C*-algebras generated by partition of unity relations between projections, where these relations are encoded in hypergraphs of this particular form. Each vertex represents a generating projection, and each hyperedge stands for a partition of unity relation.}
\label{qpg}
\end{figure}
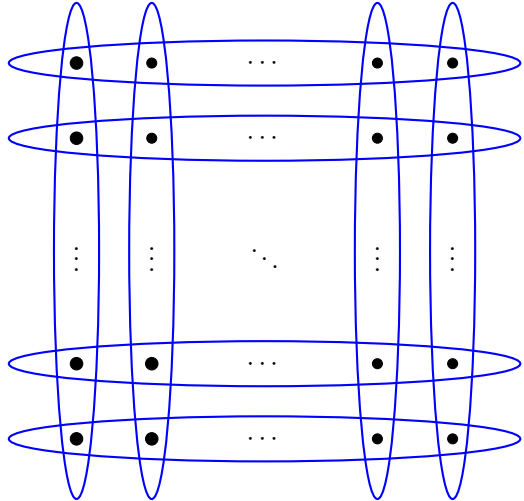

It is easy to generalize this example to arbitrary hypergraphs, which leads to the definition of free hypergraph C*-algebras. We had introduced free hypergraph C*-algebras previously in~\cite{afls} in order to study the phenomenon of quantum contextuality in the foundations of quantum mechanics. There, we had asked whether it is Turing decidable to determine whether  $C^*(H) = 0$ for a given $H$ or not. We noted that the decidability would follow if every $C^*(H)$ was residually finite-dimensional. Questions of this type are of interest due to Kirchberg's QWEP conjecture---or equivalently the Connes Embedding Problem---which can be formulated as asking whether the free hypergraph C*-algebra of the hypergraph shown in Figure~\ref{qwep} is residually finite-dimensional. Building on results of Slofstra~\cite{embedding_theorem}, we had subsequently shown that the problem ``Is $C^*(H) = 0$ for given $H$?'' is undecidable, thereby resolving the \emph{inverse sandwich conjecture} of~\cite{afls}. It follows that there are free hypergraph C*-algebras which are not residually finite-dimensional.

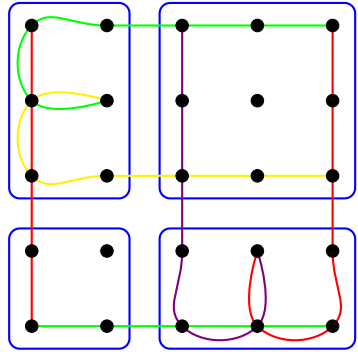
\begin{figure}
\begin{tikzpicture}
\draw[rounded corners,thick,blue] (0.7,0.7) rectangle (2.3,2.3) ;
\draw[rounded corners,thick,blue] (2.7,0.7) rectangle (5.3,2.3) ;
\draw[rounded corners,thick,blue] (0.7,2.7) rectangle (2.3,5.3) ;
\draw[rounded corners,thick,blue] (2.7,2.7) rectangle (5.3,5.3) ;
\draw[thick,red] plot coordinates { (1,1) (1,5) };
\draw[thick,red] plot [smooth,tension=.9] coordinates { (5,5) (5,3) (5,2) (5,1) (4,1) (4,2) };
\draw[thick,violet] plot [smooth,tension=.9] coordinates { (3,5) (3,3) (3,2) (3,1) (4,1) (4,2) };
\draw[thick,green] plot coordinates { (1,1) (5,1) };
\draw[thick,green] plot [smooth,tension=.9] coordinates { (5,5) (3,5) (2,5) (1,5) (1,4) (2,4) };
\draw[thick,yellow] plot [smooth,tension=.9] coordinates { (5,3) (3,3) (2,3) (1,3) (1,4) (2,4) };
\foreach \x in {1,2,3,4,5} \foreach \y in {1,2,3,4,5}
{
	\node[draw,shape=circle,fill,scale=.5,black] (e) at (\x,\y) {};
	\pgfmathtruncatemacro{\y}{1 - \x}
}
\end{tikzpicture}
\caption[]{The C*-algebra of this hypergraph\footnotemark{} is residually finite-dimensional if and only if the Connes Embedding Eroblem has a positive answer.}
\label{qwep}
\end{figure}

\footnotetext{Depending on what is easier to parse in each case, we either draw hyperedges either as closed curves that surround vertices, or as non-closed curves that connect vertices. Also, the colours that we use have no meaning beyond aiding the human eye.}

Here, we will go further and also prove that the residual finite-dimensionality itself is undecidable (Theorem~\ref{rfd}). And similarly for two other finiteness properties, namely the existence of an infinite-dimensional representation without a finite-dimensional one (Theorem~\ref{infrep}), and the existence of a tracial state (Theorem~\ref{tracial}). We also discuss the implications of such undecidability results for questions concerning independence from the ZFC axioms: for each of these undecidable properties, there are hypergraphs $H$ for which the question whether $C^*(H)$ has this property of independent of ZFC, assuming that ZFC is consistent; and likewise for any other consistent and recursively axiomatizable formal system which contains arithmetic (Corollary~\ref{zfc}). There are some treacherous subtleties hidden in the logical aspects here which we discuss in Appendix~\ref{logic}.

We also show that the class of free hypergraph C*-algebras is quite natural in itself, from two very different perspective: first, they are exactly the finite colimits of finite-dimensional commutative C*-algebras (Theorem~\ref{colimthm}); second, they are exactly the C*-algebras associated to synchronous nonlocal games~\cite{games} (Theorem~\ref{syncthm}). Also the large number of examples in Section~\ref{mainprops} underlines the naturality.

We hope that this paper will provide a unifying perspective on the different contexts in which free hypergraph C*-algebras have been used, and also that our undecidability results will indicate that the general study of free hypergraph C*-algebras may be as interesting as the study of finitely presented groups.

\subsubsection{Conventions} We work with \emph{unital} C*-algebras and \emph{unital} $*$-homomorphisms throughout, even when we do not explicitly say so; in particular, all the finitely presented C*-algebras that we consider are defined in terms of the corresponding universal property in $\Calg$, the category of unital C*-algebras and unital $*$-homomorphisms. 

\subsubsection{Caveat} As far as we can see, our free hypergraph C*-algebras have no direct relation with the more widely studied \emph{graph C*-algebras}. We welcome suggestions for a different terminology that will avoid the possibility for confusion here.

\section{The class of free hypergraph C*-algebras and alternative characterizations}
\label{mainprops}

For our purposes, a \emph{hypergraph} $H$ consists of a finite set of \emph{vertices} $V(H)$ and a collection of \emph{edges} $E(H) \subseteq 2^{V(H)}$ such that each vertex is contained in at least one edge, or equivalently $\bigcup E(H) = V(H)$. 

\begin{defn}
Given a hypergraph $H$, the \emph{free hypergraph C*-algebra} is the finitely presented C*-algebra
\beq
\label{CHdef}
	C^*(H) := \left\langle (p_v)_{v\in V(H)} \:\bigg|\;\; p_v = p^*_v = p_v^2, \quad \sum_{v\in e} p_v = 1 \;\: \forall e\in E(H) \:\right\rangle_{\Calg}
\eeq
where the subscript indicates that this is a presentation in the category of (unital) C*-algebras.
\end{defn}

Since all generators are projections, there is a preexisting bound on the norm of each $*$-polynomial in the generators, resulting in an Archimedean quadratic module in the $*$-algebra of noncommutative polynomials. This guarantees that the so defined finitely presented C*-algebra indeed exists, and it can be constructed as the completion of the complex $*$-algebra $\C[H]$ of~\eqref{CHalg} with respect to the seminorm
\[
	\| x \| := \sup_{\pi : \C[H] \to \B(\H)} \| \pi(x) \|,
\]
where $\pi : \C[H] \to \B(\H)$ ranges over all $*$-representations, and the process of taking the completion also comprises taking the quotient with respect to the null ideal of all elements $x\in\C[H]$ with $\pi(x) = 0$ first. See e.g.~\cite[p.~3]{about_connes} for a review of this type of construction.

\begin{rem}
Since free hypergraph C*-algebras are defined as universal C*-algebras given by a finite presentation, we generally do not have an explicit concrete form for them, in the sense of an explicitly described embedding into some $\B(\H)$. A notable and interesting exception is the free hypergraph C*-algebra associated to the hypergraph consisting of two non-overlapping hyperedges containing two vertices each. Equivalently, this is universal C*-algebra generated by two projections, or the group C*-algebra $C^*(\Z_2\ast\Z_2)$; see Example~\ref{du}. Since this group is amenable (e.g.~since it is isomorphic to $\Z\rtimes\Z_2$), the maximal and reduced group C*-algebras coincide, and we therefore have a concrete description available~\cite{2projections}.

However, in general we are not aware of a precise definition of what would make a finitely presented C*-algebra ``concrete''. It would be plausible to define concreteness as the existence of a countable family of states which are recursively enumerable in the sense of an algorithm which computes the value of any of these states on any $*$-polynomial in the generators up to any desired accuracy; or equivalently in terms of operators on $\ell^2(\Z)$ which represent the generators and are computable in the sense of an algorithm which returns any matrix entry with respect to the standard basis with any desired accuracy. However, we will not pursue this concreteness question any further in this work.
\end{rem}

One way think that working with the C*-algebra $C^*(H)$ is overkill, since we are merely analyzing a combinatorial structure, and might therefore just as well do something entirely algebraic. For this reason, it may also be interesting to study the analogously defined $*$-algebra
\beq
\label{CHalg}
	\C[H] := \left\langle (p_v)_{v\in V(H)} \:\bigg|\;\; p_v = p^*_v = p_v^2, \quad \sum_{v\in e} p_v = 1 \;\: \forall e\in E(H) \:\right\rangle_{\staralg}.
\eeq
Here, the subscript $\staralg$ indicates that this presentation is now understood to encode the corresponding universal property in the category of complex unital $*$-algebras, in order to keep track of the difference to~\eqref{CHdef}. So concretely, $\C[H]$ is the free noncommutative algebra over $\C$ generated by elements $p_v$ modulo the two-sided ideal generated by the elements $p_v^2 - p_v$ for all $v$ and $1 - \sum_{v\in e} p_v$ for all $e$. Upon declaring each generator to be self-adjoint, $\C[H]$ becomes a $*$-algebra, since the ideal is automatically self-adjoint.

The problem with this is that the canonical $*$-homomorphism $\C[H] \to C^*(H)$ is generally not injective, since the orthogonality relations $p_v p_w = 0$ for distinct $v,w\in e$ fail to be satisfied even in the trivial case where $H$ only consists of a single edge~\cite[p.~17]{games}, but they necessarily hold in $C^*(H)$ per Example~\ref{trivial}. We can try to fix this by quotienting by these relations as well, meaning that we consider $\C[H]/I$ for the two-sided ideal generated by the elements $p_v p_w$, where $v$ and $w$ range over all distinct $v,w\in V(H)$ occurring in some common edge. This is again a self-adjoint ideal by definition, so that $\C[H]/I$ is another $*$-algebra. One may now hope to have a purely algebraic description of a dense subset of $C^*(H)$, in the sense that the canonical $*$-homomorphism $\C[H]/I \to C^*(H)$ is injective. Unfortunately, this is not the case, in a very extreme sense:

\begin{thm}
There are hypergraphs $H$ for which $C^*(H) = 0$ despite $\C[H]/I \neq 0$.
\end{thm}

As the proof shows, this is a simple consequence of results of Helton, Meyer, Paulsen and Satriano~\cite{games} and Ortiz and Paulsen~\cite{quantum_graph_hom}. Our specific $H$ which displays this behaviour is of the form of Example~\ref{qgh}, arising from quantum graph 4-colouring.

\begin{proof}
We use the free hypergraph C*-algebras of synchronous nonlocal games and their $*$-algebras $\C[H]/I$, as discussed before the upcoming Theorem~\ref{syncthm}. Then we consider the game of 4-colouring the complete graph $K_5$. Surprisingly, a $*$-algebraic 4-colouring exists~\cite[Theorem~6.2]{games}, meaning that $\C[H]/I \neq 0$. However, no C*-algebraic 4-colouring can exist~\cite[Proposition~4.10]{quantum_graph_hom}, so that $C^*(H) = 0$.
\end{proof}

It is possible to extract from the proof a concrete description of additional algebraic relations arising between projections together with partition of unity relations, such that these relations are enforced in any Hilbert space representation, but do not follow from the standard orthogonality relations discussed above.

Thus understanding the algebra of $C^*(H)$ remains mostly an open problem:

\begin{prob}
Is there a concrete description of the kernel of the canonical $*$-homomorphism $\C[H] \to C^*(H)$? In particular, is there an algorithm to determine whether a given $x\in \Q[H]$ is in the null ideal, meaning that $\pi(x) = 0$ for all $*$-representations $\pi : \C[H] \to \B(\H)$?
\end{prob}

\subsection{Examples of free hypergraph C*-algebras}

\begin{ex}
\label{trivial}
The most trivial examples are those hypergraphs $H$ which consist of only a single edge, say with $n$ vertices. In this case, the free hypergraph C*-algebra is the C*-algebra generated by $n$ projections $p_1,\ldots,p_n$ satisfying $\sum_i p_1 = 1$. The universal property of $C^*(H)$ states exactly that, for any C*-algebra $A$, the partitions of unity in $A$, given by tuples of projections $p_1,\ldots,p_n$ satisfying $\sum_i p_i = 1$, must be in bijection with their \emph{classifying homomorphisms} $C^*(H) \to A$.

Since the orthogonality relations $p_i p_j = \delta_{ij}$ are automatic\footnote{It is well-known that if $p + q \leq 1$ for projection operators $p$ and $q$ on a Hilbert space, then $p$ and $q$ are orthogonal, meaning that $pq = 0$.}, it is easy to see that this free hypergraph C*-algebra is isomorphic to $\C^n$ with $p_i$ the $i$-th standard basis vector, or by Fourier transform equivalently to the group C*-algebra $C^*(\Z_n)$. As a degenerate special case, we also obtain the zero algebra $\C^0 \cong 0$ (associated to the hypergraph consisting of the empty edge). For $n=2$, the partitions of unity are necessarily of the form $(p,1-p)$ for a single projection $p$, so that the homomorphisms $\C^2\to A$ are precisely the classifying homomorphisms of projections in $A$.
\end{ex}

\begin{ex}
\label{du}
Given two hypergraphs $H_1$ and $H_2$, we can take their disjoint union $H_1 \amalg H_2$, which results in the free product
\[
	C^*(H_1 \amalg H_2) \cong C^*(H_1) \ast C^*(H_2),
\]
where the isomorphism is easy to see since both sides have the same universal property. (Here, $\ast$ denotes the coproduct in $\Calg$, which is the free product amalgamated over $1$.) In combination with the previous example, we obtain that all C*-algebras of the form
\[
	\C^{n_1} \ast \ldots \ast \C^{n_k} \cong C^*(\Z_{n_1} \ast \ldots \ast \Z_{n_k})
\]
are free hypergraph C*-algebras, corresponding to the class of hypergraphs having only disjoint edges.
\end{ex}

\begin{ex}
As we will see in Lemma~\ref{permanence}, imposing commutation between the generating projections of a free hypergraph C*-algebra results in another free hypergraph C*-algebra. Therefore also any \emph{graph product}~\cite{graph_prods} of finite cyclic group has a maximal group C*-algebra which is a free hypergraph C*-algebra. The proof of Proposition~\ref{qwep_rfd} gives a concrete example.
\end{ex}

\begin{ex}
\label{exsolgroup}
Continuing with the theme of maximal group C*-algebras $C^*(\Gamma)$ for finitely presented groups $\Gamma$, one also obtains a free hypergraph C*-algebra whenever $\Gamma$ is a \emph{solution group} in the sense of~\cite{solgroup}; the construction of the hypergraph is essentially that of~\cite[Lemma~10]{quantum_logic}, which shows that one also obtains a free hypergraph C*-algebra upon additionally enforcing $J = -1$.

In complete generality however, we do not know for which finitely groups $\Gamma$ it is the case that $C^*(\Gamma)$ is (isomorphic to) a free hypergraph C*-algebra.
\end{ex}

\begin{ex}
There are many nontrivial hypergraphs $H$ with $C^*(H) \cong 0$. For example if $H$ consists of three vertices such that any two of them make up an edge, then $C^*(H)$ is generated by three projections $p_1,p_2,p_3$ with $p_1 + p_2 = p_1 + p_3 = p_2 + p_3 = 1$. This implies $2 p_i = 1$ for every $i$. This leads to $0 = 1$, since this is the only way for $p_i = \frac{1}{2}$ to be a projection. Therefore $C^*(H)$ is the zero algebra. Theorem~\ref{undecidable} can be interpreted as stating that it is impossible to classify all the ways in which $C^*(H) = 0$ can occur.
\end{ex}

\begin{ex}
As outlined in the introduction, the \emph{quantum permutation group} on $n$ elements~\cite{quantum_perms,quantum_perms_survey} is the free hypergraph C*-algebra generated by projections $(p_{ij})_{i,j=1}^n$ satisfying the relations~(\ref{colsum},\ref{rowsum}). These are equivalent to postulating that the matrix of projections $U := (p_{ij})_{i,j=1}^n$ must be formally unitary,
\[
	\sum_{j=1}^n p_{ji}^* p_{jk} = \delta_{jk},
\]
which is a set of equations easily seen to be equivalent to the requirement that each row and each column of the matrix must be a partition of unity, assuming that each matrix entry $p_{ij}$ is a projection.
\end{ex}

\begin{ex}
\label{qgi}
A generalization of the previous example is that of \emph{quantum graph isomorphism}~\cite{quantum_graph_iso1,quantum_graph_iso2}. We fix a number of vertices $n$ and consider the quantum permutation group as above. Given a graph\footnote{Here, all graphs are finite and do not contain parallel edges. Although~\cite{quantum_graph_iso1,quantum_graph_iso2} only consider undirected graphs without loops, the definition and basic properties---such as the equivalence of~\eqref{am_commute} with the orthogonality relations---do not need this assumption and apply to all directed graphs, potentially with loops.} $G$ on $n$ vertices, its adjacency matrix is the matrix $A_G \in \{0,1\}^{n\times n}$ with $A_{ij} = 1$ for vertices $i,j\in G$ if and only if $i \sim j$, meaning that $i$ and $j$ are adjacent. Now given graphs $G$ and $G'$, let us consider the quantum permutation group generated by $U = (p_{ij})_{i,j=1}^n$ subject to the additional relation $U A_G = A_{G'} U$. The $(j,k)$-entry of this matrix equation is
\beq
\label{am_commute}
	\sum_{k \: : \: k \sim j} p_{ik} = \sum_{k \: : \: i \sim' k } p_{kj},
\eeq
where $\sim'$ denotes adjacency in $G'$. We now argue that this set of relations can be encoded in a set of orthogonality relations between the generating projections, so that we again obtain a free hypergraph C*-algebra by Lemma~\ref{permanence}; see also the discussion in the second half of~\cite[Section~2.1]{quantum_graph_iso2}. The crucial observation is that we have
\[
	\sum_{k \: : \: k \sim j} p_{ik} + \sum_{k \: : \: k \not\sim j} p_{ik} = 1,\qquad \sum_{k \: : \: i \sim' k} p_{kj} + \sum_{k \: : \: i \not\sim' k} p_{kj} = 1.
\]
Thanks to these relations, it follows that~\eqref{am_commute} is equivalent to the orthogonality relations
\[
	\sum_{k \: : \: k \not\sim j} p_{ik} \perp \sum_{k \: : \: i \sim' k} p_{kj}, \qquad \sum_{k \: : \: k \sim j} p_{ik} \perp \sum_{k \: : \: i \not\sim' k} p_{kj}.
\]
And each one of these in turn is equivalent to each projection appearing in the sum on the left to be orthogonal to each projection appearing in the sum on the right. This reproduces the presentation in terms of orthogonality relations as in~\cite[Section~2.1]{quantum_graph_iso2}. (See also the earlier~\cite[Lemma~3.1.1]{quantum_autos} in the case $G = G'$.) This defines the quantum isomorphism free hypergraph C*-algebra associated to $G$ and $G'$. Also~\cite[Theorems~2.4 and 2.5]{quantum_graph_iso2} follow, since the abstract C*-algebra stays the same; it's only the presentation which changes upon exchanging the relations~\eqref{am_commute} with the orthogonality relations.

\cite{quantum_graph_iso2} defines graphs $G$ and $G'$ to be \emph{quantum isomorphic} if this free hypergraph C*-algebra is nonzero. A stronger notion of quantum isomorphism was considered in the prior work~\cite{quantum_graph_iso1}, where it was required in addition that the free hypergraph C*-algebra must have a finite-dimensional representation.\footnote{This is not exactly how the definition is phrased, but is known to be equivalent~\cite[Theorem~2.1]{quantum_graph_iso2}.} These two notions of quantum isomorphism are not equivalent~\cite[Result~4]{quantum_graph_iso1}, although the latter is trivially stronger. There also is an explicit example of two graphs which are not isomorphic, but quantum isomorphic in both senses~\cite[Theorem~6.4]{quantum_graph_iso1}.

In the special case $G = G'$, the above free hypergraph C*-algebra becomes the quantum automorphism group of the graph $G$ as introduced by Banica~\cite{quantum_graph_auto}. Further specializing to $G$ being the complete or the empty graph on $n$ vertices recovers the quantum permutation group itself, since in this case the commutation $U A_G = A_G U$ holds trivially.
\end{ex}

\begin{ex}
\label{qgh}
A closely related class of examples is given by quantum graph \emph{homomorphism} rather than isomorphism. Given graphs $G$ and $G'$, the C*-algebra is quantum graph homomorphisms is the free hypergraph C*-algebra generated by a family of projections $(p_{ij})$ indexed by $i\in G$ and $j\in G'$, such that if $i_1\sim i_2$ and $j_1\not\sim' j_2$, then again we have orthogonality $p_{i_1 j_2} p_{i_2 j_2} = 0$~\cite{quantum_graph_hom}. In this case, we are not aware of an alternative formulation in terms of the adjacency matrices analogous to the one from Example~\ref{qgi}. Then again there are various possible definitions for when one says that a quantum graph homomorphism exists, depending on whether one requires the free hypergraph C*-algebra to have a finite-dimensional representation as in~\cite{variations}, or a representation with a tracial state, or just to be nonzero~\cite{quantum_graph_hom}.

In the special case where $G' = K_n$ is the complete graph, the existence of a quantum graph homomorphism $G \to K_n$ means (by definition) that $G$ is \emph{quantum $n$-colourable}, for which one can again consider the same variations~\cite{games}.
\end{ex}

The following is a variation on Kirchberg's formulation of the Connes Embedding Problem, i.e.~the QWEP conjecture~\cite{qwep}. Although---as we will see in the proof---the corresponding free hypergraph C*-algebra is really just a group C*-algebra, we nevertheless present the argument in the hope that future work will be able to reduce the size of the relevant hypergraph (Figure~\ref{qwep}) further, possibly resulting in a free hypergraph C*-algebra that is not a group C*-algebra.

\begin{prop}
\label{qwep_rfd}
The Connes Embedding Problem has a positive answer if and only if the free hypergraph C*-algebra associated to Figure~\ref{qwep} is residually finite-dimensional.
\end{prop}

\newcommand{\F}{\mathbb{F}}

\begin{proof}
The Connes Embedding Problem is equivalent to the question whether the maximal group C*-algebra $C^*(\F_2 \times \F_2)$ is residually finite-dimensional~\cite{about_connes}. Now in place of $\F_2 \times \F_2$, one can just as well take any other group which contains this one as a subgroup of finite index; such a group is e.g.~$G := (\Z_2 \ast \Z_3)^{\times 2} \cong PSL_2(\Z)^{\times 2}$.

Now the free hypergraph C*-algebra associated to the hypergraph depicted in Figure~\ref{qwep} is isomorphic to $C^*(G)$, as one can see by first drawing it in a way analogous to~\cite[Figure~7(g)]{afls}, and then noting that some of the edges are redundant, in the sense that the partition of unity relations which they implement are implied by the others via simple linear combinations~\cite[Appendix~C]{afls}.
\end{proof}

We will encounter further ways to construct free hypergraph C*-algebras in Theorems~\ref{colimthm} and~\ref{syncthm}.

\subsection{Permanence properties}

As we have seen in the above examples, one often wants to impose additional relations on the generating projections of a free hypergraph C*-algebra. We now show that one frequently obtains another free hypergraph C*-algebra.

\begin{lem}
\label{permanence}
Let $C^*(H)$ be a free hypergraph C*-algebra and $v,w\in V(H)$. Then imposing either of the following additional relations results in another free hypergraph C*-algebra:
\begin{enumerate}
\item $p_v = 0$;
\item\label{equal} $p_v = p_w$;
\item\label{orthogonal} $p_v p_w = 0$;
\item $p_v \le p_w$.
\item\label{vwcomm} $p_v p_w = p_w p_v$.
\item $p_v + p_w \leq 1$.
\end{enumerate}
\end{lem}

\begin{proof}
\begin{enumerate}
\item We simply remove vertex $v$ from $H$.
\item We add one additional vertex $u$ and the two edges $\{u,v\}$ and $\{u,w\}$, implementing the relations $p_u + p_v = 1$ and $p_u + p_w = 1$. Clearly such a projection $p_u$ exists if and only if $p_v = p_w$, and then it is unique.
\item We add one additional vertex $u$ and the edge $\{u,v,w\}$. Then a projection $p_u$ satisfying the relation $p_u + p_v + p_w = 1$ exists if and only if $p_v$ and $p_w$ are orthogonal, which means $p_v p_w = 0$. In this case, the projection $p_u$ is clearly unique.
\item Let $e\in E(H)$ be any edge with $w\in e$. Then $p_v \le p_w$ is equivalent to $p_v p_{w'} = 0$ for all $w'\in e\setminus\{w\}$, so that we can apply~\ref{orthogonal}.
\item We add four additional vertices $vw$, $\bar{v}w$, $v\bar{w}$ and $\bar{v}\bar{w}$, together with edges
\[
	\{v,\bar{v}w,\bar{v}\bar{w}\}, \qquad \{w,v\bar{w},\bar{v}\bar{w}\}, \qquad \{vw,\bar{v}w,v\bar{w},\bar{v}\bar{w}\}.
\]
Then we obtain the relations $p_v = p_{vw} + p_{v\bar{w}}$ and $p_w = p_{vw} + p_{\bar{v}w}$, which implies commutativity thanks to $p_{vw} + p_{v\bar{w}} + p_{\bar{v}w} + p_{\bar{v}\bar{w}} = 1$, which implies pairwise orthogonality. Conversely, given $p_v p_w = p_w p_v$, the new projections are uniquely determined as products, such as $p_{\bar{v}w} = (1 - p_v)p_w$.\qedhere
\item This is as in~\ref{vwcomm}, but with the vertex $vw$ removed.
\end{enumerate}
\end{proof}

\begin{lem}
\label{3wlog}
Every free hypergraph C*-algebra is isomorphic to some other free hypergraph C*-algebra $C^*(H)$, where all edges of $H$ have cardinality $3$, and any two edges intersect in at most one vertex.
\end{lem}

This is similar to the reduction of SAT to 3-SAT in computational complexity. See Remark~\ref{satisfiability} for more on this.

\begin{proof}
Let $e\in E(H)$ be an edge of the largest cardinality, and suppose that $|e| > 3$. Write this edge as a disjoint union $e = e_1 \cup e_2$, such that $|e_1| \ge 2$ and $|e_2| \ge 2$. Then introduce two new vertices $s,t$ and three new edges
\[
	e'_1 := e_1 \cup \{s\},\qquad e'_2 := e_2 \cup \{t\}, \qquad \hat{e} := \{s,t\}.
\]
The resulting partition of unity relations are $\sum_{v\in e_1} p_v + p_s = 1$ and $\sum_{v\in e_2} p_v + p_t = 1$ as well as $p_s + p_t = 1$. It is easy to see that such $p_s$ and $p_t$ exist if and only if $\sum_v p_v = 1$, which is the original partition of unity relation. We thus obtain a new hypergraph with an isomorphic C*-algebra, but with one edge of cardinality $|e|$ less. Hence iterating this process will result in an equivalent hypergraph where all edges have cardinality at most $3$.

We next show that one can achieve cardinality equal to $3$. If there is an edge of cardinality $0$, then we have $C^*(H) = 0$ anyway, in which case can be replaced by any $H'$ with $C^*(H')$ and such that all edges have cardinality $3$, which happens e.g.~for the hypergraph consisting of the four faces of a tetrahedron. Otherwise, we take all edges of cardinality less than $3$ and add new vertices until they all have cardinality $3$. For each new vertex $v$, we need to guarantee that $P_v = 0$; this can be achieved e.g.~by attaching a copy of the hypergraph depicted in Figure~\ref{gadget}.

Finally, for all pairs of edges that intersect in two vertices, say $e = \{t,u,v\}$ and $e' = \{u,v,w\}$, we remove the edge $e'$ and replace it by two new vertices $x$ and $y$ and edges $\{t,x,y\}$ and $\{x,y,w\}$. This works because under the assumption $p_t + p_u + p_v = 1$, there exist projections $p_x$ and $p_y$ satisfying $p_t + p_x + p_y = 1 = p_x + p_y + p_w$ if and only if $p_t = p_w$, or equivalently $p_u + p_v + p_w = 1$.
\end{proof}

\begin{figure}
\begin{center}
\begin{tikzpicture}[scale=1.4]
\node[draw,shape=circle,fill,scale=.5] at (0,0) {};
\node[draw,shape=circle,fill,scale=.5] (v) at (3,0) {};
\node[draw,shape=circle,fill,scale=.5] at (1.5,2.60) {};
\node[draw,shape=circle,fill,scale=.5] at (1.5,0.87) {};
\node[above left of=v,node distance=3mm] {$v$};
\draw[thick,blue] plot [smooth cycle,tension=.7] coordinates { (-.2,-.2) (3.2,-.2) (1.5,1.07) } ;
\draw[rotate around={120:(1.5,0.87)},thick,blue] plot [smooth cycle,tension=.7] coordinates { (-.2,-.2) (3.2,-.2) (1.5,1.07) } ;
\draw[rotate around={240:(1.5,0.87)},thick,blue] plot [smooth cycle,tension=.7] coordinates { (-.2,-.2) (3.2,-.2) (1.5,1.07) } ;
\end{tikzpicture}
\end{center}
\caption{Hypergraph $H$ with the property that $p_v = 0$ in $C^*(H)$.}
\label{gadget}
\end{figure}
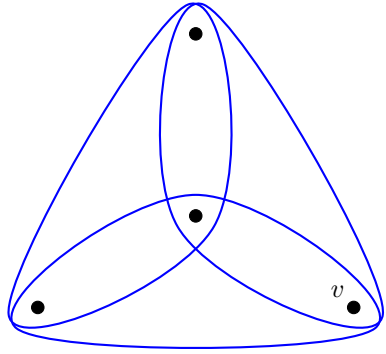

\subsection{Hypergraph C*-algebras via finite colimits}

As we saw in Example~\ref{trivial}, the finite-dimensional commutative C*-algebras $\C^n$ are trivially all free hypergraph C*-algebras.

\begin{thm}
\label{colimthm}
Up to isomorphism, the class of free hypergraph C*-algebras coincides with the class of finite colimits in $\Calg$ of finite-dimensional commutative C*-algebras.
\end{thm}

Intuitively, the reason is that every such colimit has a presentation given by a finite set of projections as generators, and relations of two types: one saying that a certain subset of generating projections should form a partition of unity; and one type saying that some two generating projections should be equal. The latter type reduces to the former by Lemma~\ref{permanence}\ref{equal}.

\begin{proof}
We first show how to express $C^*(H)$ as such a colimit for a given hypergraph $H = (V,E)$. We can present $C^*(H)$ as the unital C*-algebra with a family of generators $(p_{v,e})$ where $v\in V$ and $e\in E$ are such that $v\in e$, subject to the relations that
\begin{equation}
\label{ve}
	\sum_{v\in e} p_{v,e} = 1
\end{equation}
for all $e\in E$, as well as $p_{v,e} = p_{v,e'}$ for all $v\in V$ and $e,e'\in E$ with $v\in e,e'$. In other words, $C^*(H)$ is the unital free product of the $\C^e$ for all $e\in E$, subject to suitable identifications of their generating projections. This means that $C^*(H)$ is canonically isomorphic to the colimit C*-algebra of a diagram looking like
\[\begin{tikzcd}
	\C^2 \ar{rr} \ar{ddrr} && \C^{e} \\
	\vdots && \vdots \\
	\C^2 \ar{uurr} && \C^{e'}
\end{tikzcd}\]
where the copies of $\C^2$ on the left are indexed by the set of vertices $v\in V$, and the objects on the right by the edges $e\in E$ and are given by the C*-algebras freely generated by the above partitions of unity~\eqref{ve}. There is an arrow between $v$ and $e$ if and only if we have incidence $v\in e$. The corresponding $*$-homomorphism $\C^2 \to \C^n$ is the classifying homomorphism of $p_{v,e}$. By construction, the colimit C*-algebra of this diagram has the same universal property as $C^*(H)$ in the presentation above.

The converse direction is a bit trickier. Let $\mathsf{J}$ be a finite category indexing a diagram $D : \mathsf{J} \to \Calg$ such that every $D(J)$ for $J\in\mathsf{J}$ is commutative finite-dimensional. By definition of the colimit, homomorphisms $\colim{D}\to A$ classify tuples of homomorphisms $(\alpha_J : D(J) \to A)_{J\in\mathsf{J}}$ such that for every $f : J \to J'$, we have $\alpha_{J'} \circ D(f) = \alpha_J$. Homomorphisms $D(J) \to A$ are the same thing as partitions of unity in $A$ indexed by $\spec{D(J)}$, and the $D(f) : D(J) \to D(J')$ are determined by $\spec{f} : \spec{D(J')} \to \spec{D(J)}$. Now consider the hypergraph $H$ with vertices
\[
	V(H) := \coprod_{J\in \mathsf{J}} \spec{D(J)},
\]
and edges $E(H) := \{e_{f,v}\}$ indexed by all morphisms $f : J \to J'$ in $\mathsf{J}$ and all $v\in \spec{D(J)}$, where
\[
	e_{f,v} := \{ v \} \cup \{\: v' \in \spec{D(J')} \mid \spec{D(f)} \neq v \:\}.
\]
In particular for $f = \id{J}$ and any $v\in\spec{D(J)}$, the associated edge contains precisely the elements of $\spec{D(J)}$.

We now verify that $C^*(H)$ has the same universal property as the original colimit. By construction, into every $A \in \Calg$, $C^*(H)$ classifies families of homomorphisms $(\alpha_J : D(J) \to A)_{J\in\mathsf{J}}$ satisfying the additional compatibility requirement that for every morphism $f : J \to J'$ and element $v \in \spec{D(J)}$, we have 
\[
\alpha_J(e_v) + \sum_{v'\in J' \: : \: \spec{D(f)}(v') \neq v} \alpha_{J'}(v') = 1.
\]
But since the $\alpha_{J'}(v')$ are assumed to form a partition of unity as $v'\in \spec{D(J')}$ varies, this equation is equivalent to $\alpha_J(e_v) = \alpha_{J'}(\spec{D(f)}(v))$, or equivalently $\alpha_J = \alpha_{J'} \circ D(f)$ as described above.
\end{proof}

\begin{rem}
The above result does \emph{not} imply that the class of free hypergraph C*-algebras is closed under finite colimits, since a finite diagram of free hypergraph C*-algebras does not need to arise from a diagram of diagrams.
\end{rem}

\begin{prob}
What is an explicit example of a finite colimit of free hypergraph C*-algebras that is not itself isomorphic to a free hypergraph C*-algebra?
\end{prob}

\subsection{Hypergraph C*-algebras via synchronous nonlocal games}

A \emph{nonlocal game} for two players $A$ and $B$ is a cooperative game specified by the following data: finite sets\footnote{In general, one can also take different sets (of different cardinalities) for the two players. But as far as the standard properties of nonlocal games are concerned, there is no loss of generality in assuming that both players have the same sets of inputs and outputs.} $I$ and $O$ of \emph{inputs} and \emph{outputs} for each player, and a map
\[
	\lambda : I \times I \times O \times O \to \{0,1\}
\]
which specifies which combinations of inputs and outputs are considered winning. A round of the game consists of a referee choosing $x\in I$ and $y\in I$ and communicating the value to player $A$ and player $B$, respectively; the two players must then return respective output values $a \in O$ and $b\in O$ to the referee. The players should try to coordinate these values in such a way that the winning condition $\lambda(x,y,a,b) = 1$ is satisfied for any choice of input values $a,b\in O$. The game is \emph{winnable} if there is a strategy which achieves this. It is \emph{synchronous}~\cite[p.~4]{games} if $\lambda(x,x,a,b) = \delta_{a,b}$ holds for every input $x\in I$. Intuitively, this means that the two players must definitely agree on the answer when asked the same question, because otherwise they lose.

Nonlocal games are studied in particular in quantum information theory, where the coordination of the two players may be aided by making use of shared quantum entanglement, which allows the players to win certain nonlocal games that would not be winnable without quantum entanglement.

In general it is a difficult question to determine whether a given game is winnable with the help of quantum entanglement, and there also are various slightly different definitions of what types of quantum entanglement are allowed, concerning e.g.~whether infinite-dimensional Hilbert spaces are permitted or not~\cite{embedding_theorem}. In order to address subtleties of this kind, and also as a general tool for the study of synchronous nonlocal games,~\cite[p.~9]{games} has introduced a $*$-algebra associated to each such game. If we take its C*-completion, then it is the universal C*-algebra generated by projections $\{p_{x,a} \, : \, x\in I,\, a\in O\}$ subject to the relations
\begin{align}
\begin{split}
\label{syncgame}
	\sum_{a\in O} p_{x,a} & = 1 \qquad \forall x\in I, \\
	\qquad p_{x,a} p_{y,b} & = 0 \qquad \textrm{whenever } \lambda(x,y,a,b) = 0.
\end{split}
\end{align}
In order to write it as a free hypergraph C*-algebra, we use the proof of Lemma~\ref{permanence}\ref{orthogonal}, which results in the following hypergraph: the set of vertices is $O \times I$, plus one additional vertex for every quadruple $(x,y,a,b)$ which satisfies $\lambda(x,y,a,b) = 0$. There is one edge for every $a \in I$, and it is given by $\{ (x,a) \: : \: a \in O \}$; furthermore, there is one edge for every $(x,y,a,b)$ with $\lambda(x,y,a,b) = 0$, and it contains that extra vertex together with $(x,a)$ and $(y,b)$. It is noteworthy that the associated $*$-algebra $\C[H]/I$ defined after~\eqref{CHalg} is canonically isomorphic to the $*$-algebra associated to the game introduced at~\cite[p.~9]{games}: the synchronicity condition guarantees that the orthogonality relations which follow from the partition of unity relations in~\eqref{syncgame} are already contained in the orthogonality relations of~\eqref{syncgame}. Using this fact, it is easy to check that the two $*$-algebras enjoy equivalent universal properties.

\begin{thm}
\label{syncthm}
Up to isomorphism, the class of free hypergraph C*-algebras coincides with the class of C*-algebras of synchronous nonlocal games. Furthermore, it is enough to restrict to synchronous nonlocal games with only $|O| \le 3$ outputs.
\end{thm}

In quantum information terms, the proof is a translation of contextuality into a nonlocal game, which is conceptually not a new idea~\cite{bcs_games}. However, our translation achieves this for the hypergraph approach to contextuality~\cite{afls}, which results in a more general translation than what seems to be known.

\begin{proof}
We have shown above that the relations~\eqref{syncgame} define a free hypergraph C*-algebra.

Constructing a $3$-outcome synchronous nonlocal game for a given hypergraph $H$ is less obvious. We can assume by Lemma~\ref{3wlog} that $H$ only has edges of cardinality $3$ which pairwise intersect in at most one vertex. We enumerate the vertices in each edge $e$ as $v_{e,1}, v_{e,2}, v_{e,3}$. We then need to construct our nonlocal game such that the winning condition $\lambda$ encodes which vertices coincide, i.e.~when $v_{e,i} = v_{e',j}$ for edges $e,e'\in E(H)$ and $i,j\in \{1,2,3\}$. Although it is not clear how to implement this directly, we can equivalently encode the relations $v_{e,i} \perp v_{e',j'}$ for all $j'\neq j$ in the form of the orthogonality relations of~\eqref{syncgame} as follows. Let the set of inputs by given by the edges, $I := E(H)$, and put
\[
	\lambda(x,y,a,b) := \begin{cases} 0 & \textrm{ if } v_{x,a} = v_{y,b'} \textrm{ for } b'\neq b, \\ 
				1 & \textrm{ otherwise}, \end{cases}
\]
for the winning condition. It is easy to see that this defines a synchronous nonlocal game. Its C*-algebra is given by generating projections $p_{e,i}$ satisfying partition of unity relations $\sum_i p_{e,i} = 1$, as well as orthogonality relations $p_{e,i} \perp p_{e',j}$ whenever there is $j'\neq j$ with $v_{e,i} = v_{e',j'}$.
\end{proof}

The $\lambda$ constructed in the proof has the following alternative definition: if $a\cap b = \emptyset$, then the players always win; if $a = b$, then the players win if and only if they choose the same vertex; otherwise $|a\cap b| = 1$, in which case the players win either if they both choose the common vertex in $a\cap b$, or both do not choose it. In particular, this shows that the winning condition $\lambda$ enjoys the symmetry property $\lambda(x,y,a,b) = \lambda(y,x,b,a)$.

\begin{prob}
Is every free hypergraph C*-algebra isomorphic to one arising from quantum graph isomorphism as in Example~\ref{qgi}? From quantum graph homomorphism as in Example~\ref{qgh}?
\end{prob}

\begin{rem}
	The C*-algebras associated to synchronous nonlocal games have also been generalized to \emph{imitation games} in~\cite[Definition~4.2]{imitation}. By Lemma~\ref{permanence}\ref{orthogonal}, it is easy to see that every imitation game C*-algebra is also a free hypergraph C*-algebra. Conversely, every synchronous nonlocal game C*-algebra, and therefore also every free hypergraph C*-algebra, is an imitation game C*-algebra~\cite[Example~4.4]{imitation}. In conclusion, the class of imitation game C*-algebras also coincides with the class of free hypergraph C*-algebras.
\end{rem}

\section{Undecidable properties and independence of ZFC}
\label{sec_undec}

By the definition of $C^*(H)$ in terms of generators and relations, the unital representations $C^*(H) \to \B(\H)$ are in bijection with families of closed subspaces in $\H$ indexed by the vertices $V(H)$, such that certain subsets of the family correspond to pairwise orthogonal subspaces that span $\H$. As in~\cite{quantum_logic}, we call such a configuration of subspaces a \emph{quantum representation} or just \emph{representation} of $H$, assuming that $\H \neq 0$.

\subsection{Nontriviality}

We have $C^*(H) \neq 0$ if and only if $H$ has a quantum representation in some Hilbert space $\H$; it is enough to assume $\H$ to be separable without loss of generality. Based on results of Slofstra~\cite{embedding_theorem}, we had proven in~\cite[Corollary~11]{quantum_logic} the \emph{inverse sandwich conjecture} from~\cite{afls}:

\begin{thm}
\label{undecidable}
There is no algorithm to determine whether $C^*(H) \stackrel{?}{=} 0$ for a given $H$.
\end{thm}

In fact, the proof of~\cite[Collary~11]{quantum_logic} based on the methods of~\cite{compute_norm} shows that the problem is nevertheless semi-decidable, since there is a non-terminating algorithm. Since the undecidability proof is via reduction from the word problem for groups, which in turned is undecidable by reduction from the halting problem, we can therefore conclude that $C^*(H) \stackrel{?}{=} 0$ has the Turing degree of the halting problem.

\begin{rem}
The commutative analogue of the decision problem $C^*(H) \stackrel{?}{=} 0$ asks whether there is an assignment of a truth value to each vertex, such that exactly one vertex in every edge is true. Since all our hypergraphs are finite, this is a Boolean satisfiability problem and therefore trivially decidable (but NP-complete, by NP-completeness of 1-in-3-SAT).
\end{rem}

\begin{rem}
\label{satisfiability}
Determining whether $C^*(H) \stackrel{?}{=} 0$ can also be understood as a quantum satisfiability problem~\cite{quantum_logic,quantum_sat}. Due to the equivalence of Lemma~\ref{3wlog}, where we had shown that \emph{every} free hypergraph C*-algebra is computably isomorphic to one where all edges have cardinality $3$, determining whether $C^*(H) \stackrel{?}{=} 0$ is undecidable even when all edges of $H$ have cardinality at most $3$, which is a result of~\cite{quantum_sat}.
\end{rem}

\subsection{Residual finite-dimensionality}

In~\cite[Corollary~13]{quantum_logic}, we had also shown that there are $C^*(H)$ which are not residually finite-dimensional. In fact, we can now say something stronger:

\begin{thm}
\label{rfd}
There is no algorithm to determine whether $C^*(H)$ is residually finite-dimensional for a given $H$.
\end{thm}

This result matches the empirical difficulty of answering the Connes Embedding Problem, known to be equivalent to the residual finite-dimensionality of the free hypergraph C*-algebra of Figure~\ref{qwep}. 

\begin{proof}
Suppose that we had such an algorithm. Then we can solve the decision problem $C^*(H) \stackrel{?}{=} 0$ as follows: deciding whether $C^*(H)$ is nontrivial is equivalent to deciding whether $\|1\| = 1$ or $\|1\| = 0$ in $C^*(H)$. Using ideas from semidefinite programming and enumeration of finite-dimensional representations, it was shown in~\cite{compute_norm} that the norm of every polynomial in the generators of a finitely presented (or recursively presented) C*-algebra is computable if the C*-algebra is residually finite-dimensional. So for given $H$, we first run the assumed residual finite-dimensionality oracle. If $H$ is residually finite-dimensional, then we run the algorithm of~\cite{compute_norm} to determine whether $\|1\| = 1$ or $\|1\| = 0$. In the other case, if $H$ is not residually finite-dimensional, then we already know that it must be nontrivial. Hence we have an algorithm that decides $C^*(H) \stackrel{?}{=} 0$.
\end{proof}

In contrast to the situation with Theorem~\ref{undecidable}, we do not even know whether residual finite-dimensionality of $C^*(H)$ is semi-decidable.

\subsection{Infinite-dimensional representations}

A result which is stronger than the existence of $C^*(H)$ which is not residually finite-dimensional is the following:

\begin{thm}
\label{infrep}
There are $H$ such that $C^*(H)$ is nonzero, but has no finite-dimensional representation. Moreover, there is no algorithm to determine whether this is the case for a given $H$.
\end{thm}

As per Remark~\ref{nonstrict}, the existence of $H$ which only has infinite-dimensional representations is not new. It means that there are certain \emph{finite} configurations of closed subspaces which can only be realized in infinite-dimensional Hilbert spaces.

\begin{proof}
This is very much analogous to the proof of Theorem~\ref{rfd}. The only difference is that the finite-dimensional representations may now not be dense in the dual. However if $\|1\| = 1$, then this will still be attained in \emph{any} finite-dimensional representation, which is enough for the algorithm of~\cite{compute_norm} to work.

This proves the undecidability. The existence of $H$ for which the decision problem has a positive answer trivially follows.
\end{proof}

Again, we do not know whether the problem of Theorem~\ref{infrep} is semi-decidable.

\subsection{Existence of tracial states}

Another interesting property of free hypergraph C*-algebras is the existence of a tracial state.

\begin{thm}
\label{tracial}
There is no algorithm to determine whether $C^*(H)$ has a tracial state for a given $H$.
\end{thm}

\begin{proof}
By Theorem~\ref{syncthm} and the fact that the proof is constructive, it is enough to prove this for C*-algebras of synchronous nonlocal games. But then it is known that such a C*-algebra has a tracial state if and only if the game has a perfect commuting-operator strategy~\cite[Theorem~3.2(3)]{games}. Using the construction of synchronous games associated to binary constraint systems~\cite[Corollary~4.4(1)]{sync_bcs}, the claim follows from Slofstra's undecidability result for binary constraint systems~\cite[Corollary~3.3]{embedding_theorem}.
\end{proof}

In this case, we know that the problem is semi-decidable, and therefore also Turing equivalent to the halting problem: if there is no tracial state, then there is again a hierarchy of semidefinite programs which will detect this~\cite{trace_opt}. 

\subsection{Independence from the ZFC axioms}

The relation between undecidability and independence guaranteed by Theorem~\ref{incomplete_main} allows us to translate our undecidability results into independence results.

\begin{cor}
\label{zfc}
Let $F$ be a formal system which is recursively axiomatizable and contains elementary arithmetic. Then if $F$ is consistent, there are hypergraphs $H_1$ to $H_4$ such that each of the following sentences is independent of $F$:
\begin{enumerate}
\item $C^*(H_1) = 0$;
\item $C^*(H_2)$ is residually finite-dimensional;
\item $H_3$ has infinite-dimensional representations but no finite-dimensional ones;
\item $C^*(H_4)$ has a tracial state.
\end{enumerate}
In particular, this holds true when $F$ is given by the standard ZFC axioms.
\end{cor}

\begin{proof}
All four relevant undecidability proofs are ultimately by reduction from the halting problem. We can therefore apply Theorem~\ref{incomplete_main}.
\end{proof}

In each case, what we mean by ``independent of $F$'' is the metamathematical statement ``can neither be proven nor disproven in $F$''. Since the proof of Theorem~\ref{incomplete_main} is constructive, and so are all the computational reductions involved in reducing the halting problem to our undecidable decision problems, we conclude that it is possible in principle to write down these hypergraphs explicitly. However, we have not attempted to do so in the case where $F$ is ZFC, since we expect their descriptions to be prohibitively large, and we do not anticipate to gain much additional insight from doing so. Although it is possible that one can take $H_2$ to be the hypergraph of Figure~\ref{qwep} for ZFC, so that the Connes Embedding Problem becomes independent of the ZFC axioms, we have no particular reason to believe that this is the case.

\begin{rem}
Finding explicit examples of $H$ for which $C^*(H) \neq 0$ only has infinite-dimensional representations should be easier than finding examples of the independence from ZFC, since now one needs to retrace the computational reductions only starting with any Turing machine that does not halt, or more directly any word in the generators of a finitely presented group that does not represent the unit element. This will give an explicit example of such $H$.
\end{rem}

\subsection{Concluding remarks}

In the previous three subsections, we have considered three different properties of free hypergraph C*-algebras. For $C^*(H)\neq 0$, the negation of the second property states that $C^*(H)$ has \emph{some} finite-dimensional representation. In terms of this property, we have the following implications:
\[
	\textrm{residually finite-dimensional} \quad\Longrightarrow\quad \exists\:\textrm{finite-dimensional representation} \quad\Longrightarrow\quad \exists\:\textrm{tracial state}.
\]

\begin{rem}
\label{nonstrict}
As was pointed out to us by William Slofstra, it is not hard to see that the first implication is strict: take a solution group $\Gamma$ which is not residually finite; such a group is known to exist by Slofstra's embedding theorem~\cite{embedding_theorem}. Then $C^*(\Gamma)$ is a free hypergraph C*-algebra by Example~\ref{exsolgroup}, but is not residually finite-dimensional. However, $C^*(\Gamma)$ has a finite-dimensional representation since $\Gamma$ has at least the trivial representation.

As was pointed out to us by David Roberson, also the second implication is strict: it is known that a synchronous nonlocal game has a perfect commuting-operator strategy if and only if the associated free hypergraph C*-algebra $C^*(H)$ has a tracial state, and a perfect finite-dimensional strategy if and only it has a finite-dimensional representation~\cite[Theorem~3.2]{games}. Now there are synchronous nonlocal games for which the former type of strategy exists, but the latter does not~\cite[Corollary~4.6]{sync_bcs}.
\end{rem}

Embarrassingly, we not know whether every nontrivial free hypergraph C*-algebra has a tracial state. This is a question due to Andreas Thom:

\begin{prob}[Thom]
Is there $H$ such that $C^*(H)$ is nonzero, but has no tracial state?
\end{prob}

If the answer is positive, then the decision problems addressed by Theorems~\ref{undecidable} and~\ref{tracial} are distinct. As was pointed out to us by David Roberson, the answer is known to be \emph{negative} for all free hypergraph C*-algebras arising from quantum graph isomorphism as in Example~\ref{qgi}~\cite[Theorem~4.4]{quantum_graph_iso2}.

Our results indicate that many properties of free hypergraph C*-algebras are undecidable. This is reminiscent of the \emph{Adian--Rabin theorem} for finitely presented groups, which gives extremely general sufficient conditions for a property of group presentations to be undecidable.

\begin{prob}
Is there some analog of the Adian--Rabin theorem for free hypergraph C*-algebras?
\end{prob}

In particular, we expect a negative answer to the following interesting question:

\begin{prob}
Is there an algorithm to determine whether $C^*(H)$ is commutative for a given $H$?
\end{prob}

\newpage
\appendix
\section{A Turing machine whose halting is independent of ZFC}
\label{logic}

\newcommand{\Halt}{\mathsf{Halt}}
\newcommand{\Tind}{T_{\mathrm{ind}}}
\newcommand{\Tindhat}{\hat{T}_{\mathrm{ind}}}

This section is based on a series of fruitful discussions with Nicholas Gauguin Houghton-Larsen. We do not make any claims of originality here, but merely include this appendix for the sake of completeness. Indeed Theorem~\ref{incomplete_main} below seems to be a well-known folklore result, but not readily available in the literature in this form.

It is a standard observation that if ZFC is consistent, then there is a Turing machine $\Tind$ for which the question whether it halts on a blank tape is independent of the ZFC axioms of set theory. This is not specific to ZFC: just like G\"odel's incompleteness theorems, this holds for any sufficiently expressive and recursively axiomatizable formal system $F$. This is Theorem~\ref{incomplete_main} below.

There are two closely related standard arguments employed when trying to prove this~\cite{spectral_gap,aaron_yedi,smullyan}, which we now sketch. However, the subtlety with these arguments is that both actually require stronger assumptions than consistency.\footnote{This is certainly not a new observation, see e.g.~\href{https://mathoverflow.net/questions/130789/are-the-two-meanings-of-undecidable-related?\#comment337852\_130815}{https://mathoverflow.net/questions/130789/are-the-two-meanings-of-undecidable-related?\#comment337852\_130815}.} While the following exposition is not intended to be completely precise in the details, we keep it formal enough to illustrate the difficulties involved, which are easy to overlook. A reader not interested in these difficulties may proceed directly to the statement and proof of Theorem~\ref{incomplete_main}.

In either kind of argument, one reasons about Turing machines in $F$ by encoding every Turing machine $T$ via its description number $\hat{T}$. One also uses a natural number predicate $\Halt$ in $F$ such that $\Halt(\hat{T})$ is a sentence expressing the halting of $T$, in the sense that
\beq
\label{internalize}
	T \textrm{ halts on the blank tape} \quad\Longrightarrow\quad F \vdash \Halt(\hat{T}).
\eeq
For many reasonable $F$, such as e.g.~the ZFC axioms, it is natural to assume that also the converse implication holds: if $F$ proves $\Halt(\hat{T})$, then $T$ \emph{does} indeed halt. Similarly, if $F$ is consistent, then the contrapositive of~\eqref{internalize} shows that if proves $\lnot \Halt(\hat{T})$, then $T$ does indeed not halt. So using the converse of~\eqref{internalize} as a soundness assumption, we can reason as follows: if $F$ is such that either $F \vdash \Halt(\hat{T})$ or $F \vdash \lnot \Halt(\hat{T})$ for all $T$, then we can define $T_{\mathrm{halt\textnormal{-}solver}}$ to be the Turing machine which takes another Turing machine $T$ as input and enumerates all consequences of the axioms of $F$ until it finds a proof of $\Halt(\hat{T})$ or a proof of $\lnot\Halt(\hat{T})$. Then by the assumption that $F$ decides all instances of $\Halt(\hat{T})$, we conclude that $T_{\mathrm{halt\textnormal{-}solver}}$ always terminates, thereby solving the halting problem. Since this is absurd, one of our assumptions must have been false, meaning either that $T$ is inconsistent, or that some instance $\Halt(\hat{T})$ is independent of $F$. 

In the case of ZFC (and a suitable metatheory), the relevant soundness requirement---namely the converse of~\eqref{internalize}---would follow e.g.~from the existence of a standard model, which is a natural enough additional assumption. Nevertheless, it is also perfectly possible that $F$ is some kind of set theory like ZFC such that the set of natural numbers in $F$ contains nonstandard numbers, and that a given Turing machine $T$ halts in $F$ after nonstandard many steps; it has been argued that this is a real possibility even in the case of ZFC~\cite{indispensable}. If this happens, then the proof of its halting can clearly not be externalized, and our hypothetical Turing machine $\Tind$ does not solve the halting problem correctly on all instances.

\newcommand{\Con}[1]{\mathrm{Con}(#1)}

The other standard argument for proving the existence of a Turing machine whose halting is independent of $F$ is closely related, but requires $\omega$-consistency; although this is a weaker assumption than soundness, it is still strictly stronger that consistency. G\"odel's second incompleteness theorem is concerned with a sentence $\Con{F}$ which expresses the consistency of $F$, and states that if $F$ is consistent, then $F\not\vdash\Con{F}$. The informal statement that $\Con{F}$ expresses the consistency of $F$ means formally that $\Con{F}$ is a sentence of the form $\forall n.\Con{F,n}$, where $\Con{F,n}$ states that the first $n$ consequences of the axioms of $F$ do not contain a contradiction, in the sense that if $n\in\N$ is an external natural number, then $F \vdash \Con{F,n}$ if and only if the first $n$ consequences of $F$ do indeed not contain a contradiction. So we now take $\Tind$ to be the Turing machine which enumerates all consequences of the axioms of $F$ and halts as soon as it encounters a contradiction. Since the proof of the equivalence between halting of $\Tind$ and consistency of $F$ can be internalized, we conclude that $F$ decides $\Halt(\Tindhat)$ if and only if it decides $\Con{F}$. But since we already know that $F \not\vdash \Con{F}$, it is enough to argue that $F \not\vdash \lnot \Con{F}$ as well, which is equivalent to $F \not \vdash \exists n.\lnot \Con{F,n}$. And this is because if $F\vdash \exists n.\lnot \Con{F,n}$ were the case, then we would have an external $n\in\N$ with $F \vdash \lnot \Con{F,n}$ by $\omega$-consistency, in which case $F$ would actually be inconsistent. Thus $\Halt(\Tindhat)$ is independent of $F$.

In order to see that $\omega$-consistency is necessary in order to make this argument, it is enough to note that there are consistent $F$ such that $F\vdash \lnot \Con{F}$.\footnote{See e.g.~\href{https://mathoverflow.net/a/256862/27013}{mathoverflow.net/a/256862/27013}.} Such an $F$ also proves $\Halt(\Tindhat)$, although $\Tind$ would not actually halt, thereby violating soundness as well.

In conclusion, the standard arguments for proving the existence of a Turing machine whose halting is independent require an assumption stronger than mere consistency. Nevertheless, there is another argument which proves that consistency is enough after all. In contrast to the first argument above, it is even constructive.

\begin{thm}
\label{incomplete_main}
Let $F$ be a formal system which is recursively axiomatizable and contains elementary arithmetic. Then if $F$ is consistent, there is an explicit Turing machine $\Tind$ whose halting is independent of $F$,
\[
	F\not\vdash\Halt(\Tindhat), \qquad F\not\vdash \lnot \Halt(\Tindhat).
\]
\end{thm}

\begin{proof}
We use Kleene's symmetric version of the first incompleteness theorem~\cite[p.~69]{smullyan}. The sets
\begin{align*}
	K_1 :&= \{ \: \hat{T} \mid T \textrm{ halts and accepts on the blank tape} \:\},  \\[-4pt]
	K_2 :&= \{ \: \hat{T} \mid T \textrm{ halts and rejects on the blank tape} \:\} 
\end{align*}
are strongly separable, since simply running a Turing machine $T$ until it halts provides a proof of membership in $K_1$ or $K_2$, and this proof can be formalized in $F$.

Now for given $T$, let $T'$ be the Turing machine which runs $T$ and accepts if $T$ halts and accepts, but branches into an infinite loop if $T$ halts and rejects. Then the function $\hat{T} \mapsto \hat{T}'$ is recursive. The predicate $\hat{T} \mapsto \Halt(\hat{T}')$ expresses that for a given description number $\hat{T}$, the modified Turing machine $T'$ halts, and it strongly separates $K_1$ and $K_2$. Therefore by Kleene's symmetric version of the first incompleteness theorem, we obtain an explicit Turing machine $T$ such that $F \not\vdash \Halt(\hat{T}')$ and $F \not\vdash\lnot \Halt(\hat{T}')$. Therefore we can take $\Tind := T'$. 
\end{proof}

\newpage
\newgeometry{bottom=0.1cm}

\bibliographystyle{unsrt}
\bibliography{hypergraph_Calgebras}

\bigskip

\end{document}